\documentclass[11pt]{amsart}
\usepackage{amssymb,amsthm,amsmath,amstext}
\usepackage{mathrsfs}  
\usepackage{bm}        
\usepackage{mathtools} 
\usepackage{color}
\usepackage[bbgreekl]{mathbbol}

\usepackage[left=1.5in,top=1in,right=1.5in,bottom=1in]{geometry}

\usepackage{graphicx}

\usepackage[all]{xy}

\theoremstyle{plain}
\newtheorem{theorem}{Theorem}[section]

\newtheorem{proposition}[theorem]{Proposition}
\newtheorem{lemma}[theorem]{Lemma}

\newtheorem{problem}[theorem]{Problem}

\theoremstyle{definition}
\newtheorem{definition}[theorem]{Definition}

\theoremstyle{remark}
\newtheorem{remark}[theorem]{Remark}

\newcommand{\sheaf}[1]{\mathscr{#1}}

\newcommand{\OO}{\sheaf{O}}

\newcommand{\FF}{\sheaf{F}}

\newcommand{\sA}{\sheaf{A}}

\newcommand{\valuations}[1]{\mathsf{#1}}

\newcommand{\divisor}{\mathrm{div}}

\newcommand{\Z}{\mathbb Z}

\renewcommand{\P}{\mathbb P}
\newcommand{\Q}{\mathbb Q}
\newcommand{\Gm}{\mathbb{G}_{\mathrm{m}}}
\newcommand{\Gms}[1]{\mathbb{G}_{\mathrm{m},#1}}

\DeclareMathOperator{\Br}{\mathrm{Br}}

\DeclareMathOperator{\Spec}{\mathrm{Spec}}

\newcommand{\sep}{^{\mathrm{s}}}

\newcommand{\tensor}{\otimes}

\newcommand{\et}{\mathrm{\acute{e}t}}

\newcommand{\linedef}[1]{\textsl{#1}}
\newcommand{\Het}{H_{\et}}
\newcommand{\ur}{\mathrm{nr}}
\newcommand{\Hur}{H_{\ur}}
\newcommand{\Brur}{\Br_{\ur}}
\newcommand{\merk}{\mathrm{r}}
\newcommand{\Hr}{H_{\merk}}

\newcommand{\CH}{\mathrm{CH}}

\newcommand{\res}{\mathrm{res}}
\newcommand{\cores}{\mathrm{cor}}
\newcommand{\ord}{\mathrm{ord}}

\usepackage[backref=page]{hyperref}
\hypersetup{pdftitle={Universal triviality of the Chow group of 0-cycles and the Brauer group}}
\hypersetup{pdfauthor={Asher Auel, Alessandro Bigazzi, Christian Boehning, Hans-Christian Graf von Bothmer}}
\hypersetup{colorlinks=true,linkcolor=blue,anchorcolor=blue,citecolor=blue,letterpaper=true}

\begin{document}


\title[Universal $\CH_0$-triviality and the Brauer group]{Universal triviality of the Chow group of $0$-cycles and the Brauer group}

\author[Auel]{Asher Auel}
\address{Asher Auel, Department of Mathematics, Yale University\\
New Haven, Connecticut 06511, USA}
\email{asher.auel@yale.edu}

\author[Bigazzi]{Alessandro Bigazzi}
\address{Alessandro Bigazzi, Mathematics Institute, University of Warwick\\
Coventry CV4 7AL, England}
\email{A.Bigazzi@warwick.ac.uk}

\author[B\"ohning]{Christian B\"ohning}
\address{Christian B\"ohning, Mathematics Institute, University of Warwick\\
Coventry CV4 7AL, England}
\email{C.Boehning@warwick.ac.uk}

\author[Bothmer]{Hans-Christian Graf von Bothmer}
\address{Hans-Christian Graf von Bothmer, Fachbereich Mathematik der Universit\"at Hamburg\\
Bundesstra\ss e 55\\
20146 Hamburg, Germany}
\email{hans.christian.v.bothmer@uni-hamburg.de}

\date{\today}


\begin{abstract}
We prove that a smooth proper universally $\CH_0$-trivial variety $X$
over a field $k$ has universally trivial Brauer group.  This fills a
gap in the literature concerning the $p$-torsion of the Brauer group
when $k$ has characteristic~$p$.
\end{abstract}

\maketitle

\section{Introduction}

Our main motivation for considering the Chow group of $0$-cycles and
the Brauer group is the rationality problem in algebraic geometry.
Both groups are stable birational invariants of smooth projective
varieties that have been used to great success in obstructing stable
rationality for various classes of algebraic varieties.  The Brauer
group, and more generally, unramified cohomology groups, have been
used most notably in the context of the L\"uroth problem and in
Noether's problem, see \cite{A-M72}, \cite{Sa77}, \cite{Bogo87},
\cite{CTO}, \cite{CT95}, \cite{Bogo05}, \cite{Pey08}.  The Chow group
of $0$-cycles has gained increased attention since the advent of the
degeneration method due to Voisin~\cite{Voi15}, and further developed
by Colliot-Th\'{e}l\`{e}ne and Pirutka~\cite{CT-P16}, which unleased a
torrent of breakthrough results proving the non stable rationality of
many types of conic bundles over rational bases \cite{HKT15}, \cite{A-O16}, \cite{ABBP16},
\cite{BB16}, hypersurfaces of not too
large degree in projective space \cite{To16}, \cite{Sch18}, and many
other geometrically interesting classes of rationally connected
varieties, e.g., \cite{CT-P16}, \cite{HT16}, \cite{HPT16}.  See
\cite{Pey16} for a survey of many of these results.  This method
relies on the good specialisation properties of the \linedef{universal
triviality} of the Chow group of $0$-cycles for both
equicharacteristic and mixed characteristic degenerations.  The Brauer
group can be, and has been, used as a tool to obstruct the universal
triviality of the Chow group of $0$-cycles, principally using
equicharacteristic degenerations; our goal is to show that even for
mixed characteristic degenerations to positive characteristic $p$,
Brauer classes of $p$-primary torsion can still obstruct the universal
triviality of the Chow group of $0$-cycles.

Let $X$ be a smooth proper scheme over a field $k$, with $\CH_0(X)$
its Chow group of $0$-cycles and $\Br(X)=\Het^2(X,\Gm)$ its
(cohomological) Brauer group.  Following \cite{ACTP16} (cf.\
\cite{CT-P16}), we say that $X$ is \linedef{universally
$\CH_0$-trivial} if the degree map $\CH_0(X_F) \to \Z$ is an
isomorphism for every field extension $F/k$, and that the Brauer group
of $X$ is \linedef{universally trivial} if the natural map $\Br(F) \to
\Br(X_{F})$ is an isomorphism for every field extension $F/k$.  A
variety over a field $k$ is an integral scheme $X$, separated and of
finite type over $k$.  In this article we give a proof of the following result.


\begin{theorem}
\label{tMain}
Let $X$ be a smooth proper variety over a field $k$.  If $X$ is
universally $\CH_0$-trivial, then the cohomological Brauer group of
$X$ is universally trivial.
\end{theorem}



When $k$ has characteristic zero (more generally, for torsion in the
Brauer group prime to the characteristic), this is a result of
Merkurjev~\cite[Thm.~2.11]{Mer08}, which we review below.
Theorem~\ref{tMain} is claimed in \cite[Thm.~1.12(iii)(b)]{CT-P16},
but as it relies on the results in \cite{Mer08}, it is only proved for
torsion prime to the characteristic.  Our proof, which covers the case
of $p$-primary torsion in the Brauer group when $k$ has characteristic
$p>0$, follows a simplified version of an argument in \cite{Mer08}
utilising a pairing between the Chow group of $0$-cycles and the
Brauer group, but with several nontrivial new ingredients.

The result of Merkurjev~\cite{Mer08} is that for a smooth proper
variety $X$ over a field~$k$, the universal triviality of $\CH_0$ is
equivalent to the condition that for all Rost cycle modules $M$ (see
\cite{Ro96}) and all field extensions $F/k$, the subgroup
$M_\ur(F(X)/F) \subset M(F(X))$ of unramified elements of the function
field $F(X)$, is trivial, meaning that the natural map $M(F) \to
M_\ur(F(Y)/F)$ is an isomorphism.  In particular, taking $M$ as Galois
cohomology with finite torsion coefficients $\bbmu_l$, where~$l$ is
prime to the characteristic of $k$, the group of unramified elements
is precisely the usual unramified cohomology groups $\Hur^{i}(F(Y)/F,
\bbmu_l^{\otimes {i-1}})$, see \cite{Se97}, \cite{CTO}, \cite{CT95},
\cite{GMS03}, \cite{Bogo05}, \cite{GS06}, \cite{Pey08}. For $i=2$, one
gets the $l$-torsion in the (unramified) Brauer group
$\Br_\ur(F(Y))[l]$.  In particular, the presence of nontrivial
unramified classes obstruct the universal triviality of $\CH_0$.
Differential forms in positive characteristic, used with mixed
characteristic degenerations, provide a different type of obstruction
to the universal triviality of $\CH_0$, see \cite{To16}. For new
applications of Theorem~\ref{tMain} to the rationality problem,
see~\cite{ABBB18}.

Now suppose that $k$ has characteristic $p>0$, and we consider the
$p$-torsion $\Br(X)[p]$ in the Brauer group of $X$.  In fact,
$\Br(X)[p]$ is no longer a group of unramified elements in any graded
piece of any Rost cycle module, since residue maps are not generally
defined, see \cite[Intro.]{Mer15}, \cite[App.~A]{GMS03}.  Hence one
can no longer appeal directly to \cite{Mer08} to deduce the triviality
of $\Br(X)[p]$ under the assumption that $X$ is universally
$\CH_0$-trivial.  However, $\Br(X)[p]$ is isomorphic to the subgroup
of ``unramified classes'' in $H^2(k(X),\Z/p(1))$, where $\Z /p (j)$ is
defined as in \cite{Ka86} via the logarithmic part of the de
Rham--Witt complex.  More generally, we have the following.

\begin{problem}\label{pBr}
Let $X$ be a smooth proper variety over an algebraically closed field
$k$ of characteristic $p>0$.  Assume that $X$ is universally
$\CH_0$-trivial.  Is the subgroup of ``unramified classes'' in the
Galois cohomology group $H^i (k(X), \Z /p (j))$, i.e., those coming
from $\Het^i(\Spec(A), \Z /p (j))$ for every
discrete valuation ring $A$ with fraction field $k(X)$, trivial?
\end{problem}

Our main result gives a positive solution to Problem~\ref{pBr} when
${i=2}$ and ${j=1}$, namely, for the $p$-torsion in the Brauer group.
Note that Problem \ref{pBr} is of a similar type as many problems that
crystalline cohomology was developed for, see \cite{CL98}.

\medskip

\noindent {\bf Acknowledgments.} We would like to thank Jean-Louis
Colliot-Th\'{e}l\`{e}ne and Burt Totaro for their interest in this
work, their helpful suggestions and pointers to the literature. We
would also like to thank the Simons Foundation for hosting the
Conference on Birational Geometry August 2017 during which part of
this work was presented and several participants gave us useful
feedback.

\section{Background on Brauer groups}\label{sBrauerGroups}

All schemes below will be assumed to be Noetherian and $k$ will denote
a field unless otherwise specified.  For a scheme $X$, we will denote
by $\Br(X) = \Het^2(X,\Gm)$ the cohomological Brauer group of $X$, the
\'etale cohomology group of $X$ with values in the sheaf of units $\Gms{X}$.
For a commutative ring $A$, we also write $\Br(A)$ for the
cohomological Brauer groups of $\Spec\, A$.

The usual Brauer group of Azumaya algebras on $X$, considered up to
Morita equivalence, is a subgroup of $\Het^2(X,\Gm)$.  If $X$ is
quasi-projective over a base ring (more generally, admits an ample
invertible sheaf), then it is a result of Gabber~\cite{deJ03} that
these two groups coincide.  The analogous result over a smooth proper
variety over a field is still unknown.


If $X$ is integral with function field $K$, then restriction to the
generic point induces a map $\Br(X) \to \Br(K)$, which is an injection
when $X$ is regular (more generally, geometrically locally factorial),
see \cite[II,~Prop.~1.4]{Gro68}.

We now summarize some functorial properties of the Brauer group that
we need.

\begin{proposition}\label{pAllAuxiliary}
The cohomological Brauer group has the following properties.
\begin{enumerate}
\item\label{iPullBack} For any morphism $g : X \to Y$, there is an
induced pull-back map
\[
g^*  : \Br (Y) \to \Br (X).
\]

\item\label{iPushForward} For a finite flat morphism $f : Y' \to Y$,
there is a push-forward or corestriction map
\[
\mathrm{cor}_{Y'/Y} : \Br(Y') \to \Br(Y),
\]
also denoted by $f_*$ or $\mathrm{Nm}_{Y'/Y}$, which satisfies the
following properties:
\begin{enumerate}
\item[\textit{i)}] For a composition of finite flat morphisms $Y''\to
Y'\to Y$, we have
\[
\mathrm{cor}_{Y''/Y}=\mathrm{cor}_{Y'/Y}\circ \mathrm{cor}_{Y''/Y'}.
\]

\item[\textit{ii)}] For a finite flat morphism $f : Y' \to Y$ of
degree $d$ and $\alpha \in \Br(Y)$, we have $$f_*f^*\alpha = d\alpha.$$

\item[\textit{iii)}] Given a finite flat morphism $f : Y' \to Y$, a
morphism $g : X \to Y$, and a fiber square
\[
\xymatrix{
X' \ar[r]^{g'} \ar[d]^{f'} & Y' \ar[d]^f \\
X \ar[r]^g & Y .
}
\] 
then for any $\beta \in \Br(Y')$, we have
\[
g^* f_* \beta =  (f')_* (g')^* \beta .
\]
\end{enumerate}
\end{enumerate}
\end{proposition}

\begin{proof}
Part \ref{iPullBack}) is obvious by the contravariance of \'etale
cohomology. Part \ref{iPushForward}) is due to
Deligne~\cite[Exp.~XVII~6.3.13]{SGA4}, see also
\cite[Thm.~1.4]{ISZ11}, \cite[p.~453]{Preu13}.  For any finite flat
morphism $f : Y' \to Y$, there exists a natural homomorphism
$f_*\Gms{Y'} \to \Gms{Y}$ of sheaves of abelian groups on $Y$, cf.\
\cite[Lecture~10]{Mum66}.  Taking \'etale cohomology, we arrive at a
homomorphism
\[
\gamma : \Het^2 (Y', f_*\Gms{Y'} ) \to \Het^2(Y, \Gm ).
\]
The Leray spectral sequence for the map $f : Y' \to Y$ yields a map
\[ 
\iota  : \Het^2 (Y, f_*\Gms{Y'} ) \to \Het^2 (Y', \Gm),
\]
which is an isomorphism because $R^i f_* \Gms{Y'} = 0$ for all $i>0$
as $f$ is finite. Putting $\mathrm{cor}_{Y'/Y} := \gamma \circ
\iota^{-1}$, we obtain a corestriction map on Brauer groups. This has
property \textit{i)} by construction, \textit{ii)} since the
composition $\Gms{Y} \to f_*\Gms{Y'} \to \Gms{Y}$ with the unit of adjunction is
multiplication by the degree of $f$, and \textit{iii)} by the
compatibility of the norm map with base change, of the Leray spectral
sequence with pull-backs, and of the Leray spectral sequence with
change of coefficients, see also \cite[\S4]{Preu13}.
\end{proof}

\begin{definition}\label{dValueBrauerOnPoint}
If $i : x \hookrightarrow X$ is the inclusion of a point
with residue field $k (x)$ into $X$, then for an $\alpha \in \Br(X)$
we denote $i^* (\alpha ) \in \Br(k (x))$ simply by $\alpha (x)$ and
call it the value of $\alpha$ at $x$.
\end{definition}

\begin{definition}\label{dNormRatFunct}
Let $\varphi : X \to Y$ be a finite surjective morphism of
$k$-varieties, and let $f\in k(X)$ and $g \in k(Y)$ be rational
functions.  We denote by $\varphi^* (g)$ the pull-back of the rational
function $g$ to $X$, and by $\mathrm{Nm}_{X/Y}(f)$, or $\varphi_*(f)$,
the norm of the rational function $f$ to $Y$. The later is defined as
follows: via $\varphi^*$, there is a finite extension of fields $k(X)
/ k(Y)$, and $\mathrm{Nm}_{X/Y}(f)$ is the determinant of the
endomorphism of the $k(Y)$-vector space $k(X)$ that is given by
multiplication by
$f$. 
\end{definition}

In applications, when one wants to determine $\Br(X)$ for some smooth
projective $k$-variety $X$, one is frequently only given some singular
model of $X$ a priori: it thus becomes desirable, especially since
some models of such $X$ can be highly singular and difficult to
desingularise explicitly, to determine $\Br(X)$ in terms of data
associated to the function field $k(X)$ only. This is really the idea
behind unramified invariants as for example in \cite{Bogo87},
\cite{CTO}, \cite{CT95}. We have an inclusion
\begin{displaymath}\label{fInclusionBrauer}
\Br(X) \subset \Br(k(X))
\end{displaymath}
by \cite[II,~Cor.~1.10]{Gro68}, given by pulling back to the generic
point of $X$. One wants to single out the classes inside $\Br(k(X))$
that belong to $\Br(X)$ in valuation-theoretic terms. Since the basic
reference \cite{CT95} for this often works under the assumption that
the torsion orders of the classes in the Brauer group be coprime to
the characteristic of $k$, we state and prove below a result in the
generality we need here, although most of its ingredients are
scattered in the available literature.

Basic references for valuation theory are \cite{Z-S76} and \cite{Vac06}. Valuation here without modifiers such as ``discrete rank $1$" etc. means a general Krull valuation.

\begin{definition}\label{dValBrauer}
Let $X$ be a smooth proper variety over a field $k$ and let
$\valuations{S}$ be a subset of the set of all Krull valuations of the
function field $k(X)$ of $X$. All the valuations we will consider
below will be geometric in the sense that they are assumed to be
trivial on the ground field $k$. For $v\in \valuations{S}$, we denote by
$A_v\subset k(X)$ the valuation ring of $v$. Then we denote by
$\Br_{\valuations{S}}(k(X)) \subset \Br(k(X))$ the set of all those
Brauer classes $\alpha\in \Br(k(X))$ that are in the image of the
natural map $\Br(A_v) \to \Br(k(X))$ for all $v \in
\valuations{S}$. Specifically, we will consider the following examples of
sets $\valuations{S}$.
\begin{enumerate}
\item The set $\valuations{DISC}$ of discrete rank $1$ valuations of
$k(X)$ with fraction field $k(X)$.

\item The set $\valuations{DIV}$ of all divisorial valuations of
$k(X)$ corresponding to some prime divisor $D$ on a model $X'$ of
$k(X)$, where $X'$ is assumed to be generically smooth along $D$.

\item The set $\valuations{DIV/}X$ of all divisorial valuations of $k(X)$
corresponding to a prime divisor on $X$.
\end{enumerate}
We denote the corresponding subgroups of $\Br(k(X))$ by 
\[
\Br_{\valuations{DISC}}(k(X)),\: \Br_{\valuations{DIV}}(k(X)),\: \Br_{\valuations{DIV/}X}(k(X))
\]
accordingly. In addition, we define 
$$
\Br_{\valuations{LOC}}(k(X))
$$ 
as those classes in $\Br(k(X))$ coming from $\Br(\mathcal{O}_{X,x})$
for every (scheme-)point $x\in X$.
\end{definition}

Note the containments $\valuations{DISC} \supset \valuations{DIV} \supset \valuations{DIV/}X$,
which are all strict in general: for the first, recall that divisorial
valuations are those discrete rank $1$ valuations $v$ with the
property that the transcendence degree of their residue field is $\dim
X -1$ \cite[Ch.~VI,~\S14,~p.~88]{Z-S76}, \cite[\S1.4,~Ex.~5]{Vac06};
and that there are discrete rank $1$ valuations that are not
divisorial, for example, so-called analytic arcs \cite[Ex.~8(ii)]{Vac06}.

\begin{theorem}\label{tComparisonBrauer}
Let $X$ be a smooth proper variety over a field $k$.  Then all of
the natural inclusions
\[
\Br(X) \subset \Br_{\valuations{DISC}}(k(X)) \subset \Br_{\valuations{DIV}}(k(X)) \subset \Br_{\valuations{DIV/}X}(k(X))
\]
are equalities.
\end{theorem}

To prove this, we need two preliminary results.  The first is a purity
result for the cohomological Brauer group of a variety over a field.

\begin{theorem}\label{tGabber}
Let $V$ be a smooth $k$-variety, and let $U \subset V$ be an open
subvariety such that $V-U$ has codimension $\ge 2$ in $V$. Then the
restriction $\Br(V) \to \Br(U)$ is an isomorphism.
\end{theorem}
\begin{proof}
The proof starts with a reduction to the case of the punctured spectrum
of a strictly Henselian regular local ring of dimension $\geq 2$.  For
torsion prime to the characteristic of $k$, this follows from the
absolute cohomological purity conjecture, whose proof, due to Gabber,
appears in \cite{Fuji02} or \cite{ILO14}.  For the $p$ primary torsion
when $k$ has characteristic $p$, this follows from
\cite[Thm.~2.5]{Ga93}.  See also \cite[Thm.~5]{Ga04} and its proof.
Recently, purity for the cohomological Brauer group has been
established in complete generality over any scheme \cite{Ces17}.
\end{proof}

The second is a standard Meyer--Vietoris exact sequence for \'etale
cohomology.

\begin{theorem}\label{tMayerVietorisEtale}
Let $V$ be a scheme. Suppose that $V=U_1\cup U_2$ is the union of two
open subsets. For any sheaf $\FF$ of abelian groups in the \'{e}tale
topology on $V$ there is a long exact cohomology sequence
\begin{align*}
0 & \to \Het^0 (V, \FF ) \to \Het^0 (U_1, \FF ) \oplus \Het^0 (U_2, \FF ) \to \Het^0 (U_1\cap U_2, \FF ) \\
   & \to \Het^1 (V, \FF ) \to \Het^1 (U_1, \FF ) \oplus \Het^1 (U_2, \FF ) \to \Het^1 (U_1\cap U_2, \FF ) \to \dots 
\end{align*}
which is functorial in $\FF$.
\end{theorem}

\begin{proof}
See \cite[Ex. 2.24, p. 110]{Mi80}.
\end{proof}

\begin{proof}[Proof (of Theorem \ref{tComparisonBrauer})]
Note that to ensure that one has the inclusion $\Br(X) \subset
\Br_{\valuations{DISC}}(k(X))$ one uses the valuative criterion for
properness so that every valuation on $k(X)$ is centered at a point of
$X$. The inclusion $\Br(X) \subset \Br_{\valuations{DIV/}X}(k(X))$
holds regardless of properness assumptions, and the inclusions
$\Br_{\valuations{DISC}}(k(X)) \subset \Br_{\valuations{DIV}}(k(X))
\subset \Br_{\valuations{DIV/}X}(k(X))$ come immediately from the
definitions.

We prove that every class $\alpha$ in $\Br_{\valuations{DIV/}X}(k(X))$
belongs to $\Br(X)$. Any class $\alpha$ in $\Br(k(X))$ can be
represented by a class, denoted $\alpha_V$, in $\Br(V)$ where
$V\subset X$ is open with complement a union of prime divisors $D_i$
on $X$. Moreover, as $\alpha$ is in the image of $\Br(\mathcal{O}_{X,
\xi_i})$ for $\xi_i$ the generic point of $D_i$ and all $i$, we have
that there are open subsets $U_i$ of $X$ such that $U_i\cap D_i\neq
\emptyset$ and there exist classes $\alpha_{U_i}\in \Br(U_i)$ whose
images in $\Br(k(X))$ agree with $\alpha$. By
Theorem~\ref{tMayerVietorisEtale}, using the cohomological description
of the Brauer group, we get that there exist an open subset $U\subset
X$ with complement $X \backslash U$ of codimension at least $2$ and a
class $\alpha_U\in \Br(U)$ inducing $\alpha$ in $\Br(k(X))$. Note that
this uses Grothendieck's injectivity result that Brauer classes on a
regular integral scheme agreeing at the generic point agree everywhere
to achieve the gluing: to use Theorem \ref{tMayerVietorisEtale}
repeatedly, we have to ensure that at each step the Brauer classes
given on the two open sets $\Omega_1=V\cup U_1\cup\dots \cup U_{i-1}$
and $\Omega_2=U_i$ agree on all of $\Omega_1\cap \Omega_2$ and this
follows from Grothendieck's result since we know they agree in
$\Br(k(X))$.

Finally, once we have extended $\alpha$ to the open subset $U$ with
$\mathrm{codim}(X\backslash U, X) \ge 2$, Theorem~\ref{tGabber} shows
that $\alpha$ comes from $\Br(X)$.
\end{proof}

\begin{remark}\label{rUnramifiedBrauer}
In the setting of Theorem~\ref{tComparisonBrauer}, we will agree to
denote the group $\Br_{\valuations{DIV}}(k(X))$ by $\Br_\ur(k(X))$ and
call this the unramified Brauer group of the function field $k(X)$. We
will also use this notation for potentially singular $X$ in positive
characteristic. According to \cite{Hi17}, a resolution of singularities
should always exist, but we do not need this result: in all our
applications we produce explicit resolutions $\widetilde{X}$, and then
we know $\Br_\ur(k(X))=\Br(\widetilde{X})$.
\end{remark}

\section{A variant of Weil reciprocity and a pairing} \label{sWeilAndPairing}

Let $V$ be a proper variety over a field $k$ (not necessarily
algebraically closed in the sequel).  Let $Z_0(V)$ be the group of
$0$-cycles on $V$ and $\CH_0(V)$ the Chow group of $0$-cycles on $V$
up to rational equivalence.  We first define a pairing
$$
Z_0(V) \times \Br(V) \to \Br(k)
$$
as follows: for a $0$-cycle $z = \sum_i a_i z_i$ and a Brauer class
$\alpha \in \Br(V)$, define
$$
\langle z, \alpha \rangle = \alpha(z) = \sum_i a_i \cores_{k(z_i)/k}(\alpha(z_i))
$$
where $a_i$ are integers, $z_i$ are closed points of $V$ with residue
fields $k(z_i)$, and $\cores$ is the corestriction map. Thus $\langle z, \alpha \rangle$ and $\alpha(z)$ are simply alternative notation for this pairing which are convenient in different settings. 
This pairing is clearly bilinear.  The main result of this section is
the following.

\begin{proposition}\label{pFactorise}
Let $X$ be a smooth proper variety over a field $k$.  The pairing on
$0$-cycles defined above factors through to a pairing
$$
\CH_0(X) \times \Br(X) \to \Br(k).
$$
\end{proposition}

We need to show that rationally equivalent $0$-cycles on $X$ have the
same values when paired with Brauer classes. The $0$-cycles on $X$
rationally equivalent to zero are sums of cycle of the form
$\pi_*(z)$, where $D \subset X$ is a curve with normalization
\[
\pi  : C \to D
\]
and $z= \mathrm{div}(f)$  for a rational function $f \in k(C)^*$.

\begin{lemma}\label{lAdjunction1}
If $\pi  : C \to D$ is any finite morphism of proper curves, $\alpha \in \Br(D)$, $z \in Z_0 (C)$, then 
\[
\langle z, \pi^* (\alpha ) \rangle = \langle \pi_* (z), \alpha \rangle .
\]
\end{lemma}

\begin{proof}
Since the pairing is bilinear and $\pi_*$ linear, it suffices to prove this for the case where $z \in C$ is a single closed point with residue field $k(z)$. The image $w=\pi (z)$ (of the map $\pi$ on the underlying point sets of $C$ and $D$) is a closed point of $D$ and we have a commutative diagram
\[
\xymatrix@C=36pt{
 & \Spec\, k(z) \ar[ld]_{\epsilon_z} \ar[d]^{\pi_z} \ar@{^{(}->}[r]^(.63){\iota_z} & C \ar[d]^{\pi} \\
\Spec\; k & \Spec\, k(w)\ar[l]^{\epsilon_w} \ar@{^{(}->}[r]^(.63){\iota_w} & D
}
\]
The left hand side $\langle z, \pi^* (\alpha ) \rangle$ is nothing but
$\epsilon_{z,*} (\iota_z^* \pi^* \alpha )$, which we can rewrite as
\[
\epsilon_{z, *} \iota_z^* \pi^* \alpha  = \epsilon_{w, *}\pi_{z, *}\pi_z^*\iota_w^* \alpha = [k(z) : k(w)] \epsilon_{w, *} \iota_w^* \alpha = \langle \pi_* (z), \alpha \rangle
\]
noting that $\epsilon_{z,*} = \epsilon_{w,*}\pi_{z,*}$ and that $\pi_{z,
*}\pi_z^*$ is multiplication by $[k(z):k(w)]$ by
Proposition~\ref{pAllAuxiliary}\ref{iPushForward}, and since $\pi_*
(z) = [k(z) : k(w)] w$.
\end{proof}

We point out that the previous Lemma~\ref{lAdjunction1} holds more
generally for any finite morphism of smooth proper varieties. Now, the
proof of Proposition \ref{pFactorise} will be complete once we
establish the following.

\begin{proposition}\label{pWeilRecAnalogue}
Let $C$ be a normal proper curve over $k$.  Then for any rational
function $f \in k(C)^*$ and any Brauer class $\alpha \in \Br(C)$, we
have that $\alpha(\divisor(f))=0$.
\end{proposition}

Note that if $\alpha$ has order $l$ coprime to the characteristic of $k$, then one can write  
$$ 
\alpha(\divisor(f)) = \sum_{P \in C^{(1)}} \cores_{k(P)/k}(\ord_P(f) \alpha(P))
= \sum_{P \in C^{(1)}} \cores_{k(P)/k} \partial_P \bigl( (f) \cup \alpha \bigr)
$$
where the sum is over all codimension 1 points of $C$, where $\ord_P$
is the order of vanishing, and where $\partial_P$ is the residue map
associated to the discrete valuation on $k(C)$ determined by
$P$. Hence, Proposition \ref{pWeilRecAnalogue} follows from the Weil
reciprocity formula in Rost \cite[Prop. 2.2 (RC)]{Ro96}. See also
\cite[Proposition~7.4.4]{GS06} for a general Weil reciprocity law in
Milnor $K$-theory. For torsion a power of $\mathrm{char}(k)$, the
theory of residue maps breaks down and only some remnants remain
\cite[App.~A,~p.~152ff]{GMS03}; the proof in \cite[p.~340--341]{Ro96}
heavily relies on residues as well as co-residues that we do not know
how to define in this situation, so we resort to a more elementary
approach to prove the special case of Weil reciprocity needed in
Proposition~\ref{pWeilRecAnalogue}. In fact, the proof is modeled on
the proof of classical Weil reciprocity on a smooth projective curve
$C$: for nonzero rational functions $f, g$ with disjoint zero and
polar sets $f (\mathrm{div}(g)) = g (\mathrm{div}(f))$, see
\cite[p.~283]{ACGH85}.

\begin{proof}[Proof (of Proposition \ref{pWeilRecAnalogue})]
The pairing above induces a bilinear pairing
\[
k(C)^* \times \Br(C) \to \Br(k), \quad (f, \alpha) \mapsto \langle f,  \alpha \rangle
\]
by writing $\mathrm{div}(f) = \sum_i a_i z_i$ and defining $\langle f, \alpha \rangle = \sum_i a_i \cores_{k(z_i)/k}(\alpha(z_i))$ as above. We want to show that this pairing is trivial on $k(C)^* \times \Br(C)$.

\textbf{Step 1.} We prove Proposition \ref{pWeilRecAnalogue} for $C =
\P^1_k$. In that case, $\alpha$ is induced from $\Br(k)$; more
precisely, if $\omega : \P^1_k \to \Spec(k)$ is the structure
morphism, $\alpha = \omega^* (\alpha')$ from some $\alpha' \in
\Br(k)$: indeed, a theorem due to Grothendieck says that
$\Br(\P^1_k)=\Br(k)$ for any field; alternatively, by \cite[Prop. 3.4
(4)]{Mer15}, we have that the natural morphism $H^2(k, \Q/\Z (1)) \to
\Hur^2 (k(\P^1_k), \Q/\Z (1))$ is an isomorphism, and by \cite[\S
2.1]{Mer15}, $H^2 (F, \Q/\Z (1)) = \Br(F)$ for any field. Moreover,
$\Br(X)$ injects into $\Brur(k(X))$ for any regular integral
$k$-variety \cite[II,~Cor.~1.10]{Gro68}, with the definition of
unramified elements as in \cite[\S 2.4]{Mer15}.

Now if $z_i$ is the image of $\psi_i  : \Spec(k(z_i)) \to X$, then the pull-back of $\alpha'$ via the composition
\[
\Spec(k(z_i)) \stackrel{\psi_i}{\longrightarrow} X \stackrel{\omega}{\longrightarrow} \Spec(k)
\]
followed by the corestriction $\cores_{k(z_i)/k} : \Br(k(z_i)) \to \Br(k)$ is nothing but $[k(z_i):k] \alpha'$. Hence,
\[
\langle f, \alpha \rangle = \sum_i a_i \cores_{k(z_i)/k}(\alpha(z_i)) = \alpha' \sum_i [k(z_i):k] a_i = \deg(\mathrm{div}(f)) \alpha' =0
\]
since the degree of a principal divisor is zero.

\textbf{Step 2.} Reduction of the general case to $\P^1_k$. Given a nonconstant $f\in k(C)^*$ (for constant $f$'s the assertion to be proved is trivial), it defines a surjective finite flat morphism $C \to \P^1_k$, which we will denote by $\varphi_f$. Moreover, letting $g=\mathrm{id}_{\P^1_k}\in k(\P^1_k)$ we have the very stupid but useful equality $\varphi^*_f (g) = f$. 
Now by Lemma \ref{lAdjointness}, proven below, to conclude, we can simply compute
\[
\langle f, \alpha \rangle = \langle \varphi_f^* (g), \alpha \rangle = \langle g, (\varphi_f)_* (\alpha ) \rangle = 0,
\]
the last equality by Step 1.
\end{proof}

\begin{lemma}\label{lAdjointness}
For a finite flat covering of normal proper curves $\varphi : C \to D$
over $k$, a $0$-cycle $z \in Z_0 (D)$, and a Brauer class $\alpha \in
\Br(C)$, we have
\[
\langle \varphi^* (z) , \alpha \rangle = \langle z, \varphi_* (\alpha ) \rangle 
\]
where $\varphi_* (\alpha ) = \mathrm{Nm}_{C/D}(\alpha )$ is the norm
map of Proposition~\ref{pAllAuxiliary} and $\varphi^* (z)$ the
pull-back of divisors from Definition~\ref{dNormRatFunct}.
\end{lemma} 

\begin{proof}
We divide the proof into steps for greater transparency. By the bilinearity of the pairing and linearity of $\varphi^*$, we can reduce to the case when $z$ is a closed point in $D$ with residue field $k(z)$, which we will assume in the sequel. 

\textbf{Step 1.}  Applying
Proposition~\ref{pAllAuxiliary}~\ref{iPushForward}~\textit{iii}, to
\[
\xymatrix{
C_z \ar[r] \ar[d]^{\varphi_z} & C \ar[d]^{\varphi}\\
\Spec\, k(z) \ar[r] & D
}
\]
where $C_z$ is the scheme-theoretic fiber over $z$, we see that it suffices to prove the equality in Lemma \ref{lAdjointness} for $\varphi$ replaced by $\varphi_z$, $\alpha$ replaced by $\alpha_z$, the restriction of $\alpha$ to $C_z$. In the following we drop the subscripts $z$ again for ease of notation.

\medskip

\textbf{Step 2.} Abbreviating $k(z)$ to $L$, the fiber $C_z$ in Step 1 is a disjoint union of schemes of the form $Y_{x}:=\Spec \, \OO_{C, x}/((\pi_x)^{n_{x}})$ where $x\in C$ is a closed point with local ring the discrete valuation ring $ \OO_{C, x}$, with uniformiser $\pi_x$, $n_x\in \mathbb{N}$, and residue field $k (x)$ some finite extension field of $L$. Here $x$ runs over the closed points lying above $z$. By the bilinearity of the pairing and linearity of $\varphi_*$ again, it thus suffices to prove: given
\begin{gather}\label{dCurveCovering}
\xymatrix{
& Y_{x}\ar[ld] \ar[d]^{\varphi}  & ~\Spec \, \OO_{C, x}/(\pi_x)= \Spec(k (x)) \ar[ld]^{\psi} \ar@{_{(}->}[l]_(.75){\mathrm{red}}\\
\Spec\, k & \Spec\, L \ar[l] &   
}
\end{gather}
and a Brauer class $\alpha$ in $\Br(Y_x)$, then $\varphi_* (\alpha )$ is the same as $n$ times $\psi_* (\mathrm{red}^* (\alpha))$. Here $\mathrm{red}$ is the reduction map, and $\psi$ is induced by the inclusion of fields $L \subset k (x)$. We will also sometimes find it convenient to write $(Y_x)_{\mathrm{red}}$ for $\Spec(k (x))$ in the argument below. 

\medskip

\textbf{Step 3.} We prove the remaining assertion of Step 2. To ease
notation, we drop the subscript $x$, and denote $\OO_{C, x}$ simply by
$\OO$. Let $\alpha$ be a Brauer class in $\Br(\OO/(\pi^n))$, for the
given discrete valuation ring $\OO$ (which contains the field $L$)
with residue field $\kappa =\OO/(\pi)$ being a finite extension of
$L$, $[\kappa : L]=d$. Let $\bar{b}_1, \dots , \bar{b}_d$ be an
$L$-basis of $\kappa $, and let $b_1, \dots , b_d$ be lifts of these
basis elements to $\OO$. Then the elements
\[
b_1, \dots , b_d, \; \pi b_1, \dots, \pi b_d, \dots , \pi^{n-1}b_1, \dots , \pi^{n-1}b_d
\]
are an (ordered) $L$-basis of $\OO/(\pi^n)$. Abbreviate
$\sA=\OO/(\pi^n)$ and write an element $a\in \sA$ in that given
basis:
 \begin{align*}
a= & a_{0,1}b_1 +  \dots + a_{0, d}b_d \\
  & + (a_{1,1}b_1 +  \dots + a_{1, d}b_d)\pi  \\
 & \vdots \\
 & + (a_{n-1, 1}b_1 + \dots + a_{n-1, d}b_d) \pi^{n-1}.
\end{align*}
Let $\kappa=\bar{\sA}$ be the reduction of $\sA$, and let $\bar{a}\in \bar{\sA}$ be the image of $a$, which can consequently be written as 
\[
\bar{a}= a_{0,1}\bar{b}_1 +  \dots + a_{0, d}\bar{b}_d 
\]
in the (ordered) $L$-basis $\bar{b}_1, \dots , \bar{b}_d$ of $\bar{\sA}$. 

Denote by $m_a$ and $m_{\bar{a}}$ the $L$-linear maps in $\sA$ and
$\bar{\sA}$ given by multiplication by $a$ and $\bar{a}$. If in the
chosen ordered basis above, $m_{\bar{a}}$ has a matrix $M$, then in
the chosen ordered basis of $\sA$ the map $m_a$ will be represented by
a block lower triangular matrix of the form
\[
\begin{pmatrix}
M         & 0          & 0 &  \dots & 0 \\
N_{21} &         M & 0 & \dots & 0 \\
N_{31} & N_{32} & M & \dots & 0\\
\vdots & \vdots & \vdots & \ddots & \vdots \\
N_{n1} & N_{n2} & \dots & N_{n, n-1} & M 
\end{pmatrix}
\]
where $M$ and all matrices $N_{ij}$ are $d\times d$ matrices with entries in $L$. In particular, 
\[
\mathrm{Nm}_{\sA/L}(a) = \left( \mathrm{Nm}_{\bar{\sA}/L}(\bar{a})\right)^n .
\]
We can rephrase this by saying that in the set-up of the diagram (\ref{dCurveCovering}), we have that the norm map
\begin{gather}\label{fNormMapCurveCovering}
N : \varphi_* \Gms{Y_x} \to \Gms{\Spec(L)}
\end{gather}
can be factored into three maps:
\begin{gather}\label{fThreeMaps}
\xymatrix{
\varphi_* \Gms{Y_x} \ar[r]^(.35){\varrho} & \varphi_* \left(  \mathrm{red}_*\Gms{(Y_x)_{\mathrm{red}}} \right) \ar@{=}[r] & \psi_* \left(  \Gms{(Y_x)_{\mathrm{red}}} \right) \ar[r]^{\bar{a}\mapsto \bar{a}^n}  & \psi_* \left( \Gms{(Y_x)_{\mathrm{red}}} \right)\ar[d]^{\bar{N}} \\
      & & & \Gms{\Spec(L)}
}
\end{gather}
where $\varrho$ is induced by the reduction map $\mathrm{red}$ and $\bar{N}$ is again a norm map. Applying $\Het^2 (\Spec(L), - )$ to this sequence of homomorphism of sheaves of (multiplicative) abelian groups in the \'{e}tale topology, and noting that  
\[
\Het^2 (\Spec(L), \varphi_* \left( \Gms{Y_x} \right) ) \simeq \Br(Y_x)
\]
and
\[
\Het^2 \left(\Spec(L), \psi_* \left(  \Gms{(Y_x)_{\mathrm{red}}} \right)\right) \simeq \Br((Y_x)_{\mathrm{red}})
\]
by the Leray spectral sequence and the finiteness of $\varphi$ and $\psi$, as in the proof of Proposition \ref{pAllAuxiliary}, we obtain that for a Brauer class $\alpha$ in $\Br(Y_x)$, the norm $\varphi_* (\alpha )$ is indeed the same as $n$ times $\psi_* (\mathrm{red}^* (\alpha))$.
\end{proof}

\section{Proof of Theorem \ref{tMain}}\label{sProofThmMain}

Let $X$ be a smooth proper variety over a field $k$. Denote by
${\omega : X \to \Spec(k)}$ the structure morphism.  For any field
extension $F/k$, we have the above pairing
$$
\CH_0(X_F) \times \Br(X_F) \to \Br(F)
$$
and we will mostly be interested in the case when $F=k(X)$.

Let $\eta$ be the generic point of $X$, considered as a closed point
in $X_{k(X)}$.  Then ``pairing with $\eta$'' defines a map $\Br(X) \to
\Br(k(X))$ by pulling $\alpha \in \Br(X)$ back to $\Br(X_{k(X)})$ and
then pairing with $\eta$, thought of as an element in
$\CH_0(X_{k(X)})$.  This map coincides with the usual map restricting
a Brauer class to its own function field, which is injective for $X$
smooth by \cite[II,~Cor.~1.10]{Gro68}.

Now we assume that $X$ is universally $\CH_0$-trivial and proceed to
prove Theorem~\ref{tMain}, namely that $\Br(k)=\Br(X)$.  The same
proof shows that $\Br(F)=\Br (X_F)$ for any extension field $F/k$,
i.e., that the Brauer group of $X$ is universally trivial.  By
assumption, $X$ has a $0$-cycle $z_0$ of degree~1.

We can assume that the support of $z_0$ consists of closed points
whose residue fields are separable extensions of $k$.  Indeed,
answering a question of Lang and Tate from the late 1960s, it is a
result of Gabber, Liu, and
Lorenzini~\cite[Theorem~9.2]{gabber_liu_lorenzini:index} that a
regular generically smooth nonempty scheme of finite type over $k$
(e.g., our smooth proper variety $X$ over $k$) admits a 0-cycle of
minimal positive degree supported on closed points with separable
residue fields.  In our case, the minimal positive degree of a
$0$-cycle on $X$ is $1$.

Let $\alpha \in \Br(X)$ and define $\alpha_0 = \alpha(z_0) \in
\Br(k)$.  Then $(\alpha - \omega^*(\alpha_0))(z_0) = 0$ because
$(\omega^*(\alpha_0)) (z_0) = \deg (z_0) \cdot \alpha_0 = \alpha_0$.

Let $z_0' \in \CH_0(X_{k(X)})$ be the 0-cycle that $z_0$ determines on $X_{k(X)}$. 
 Since $X$ is
universally $\CH_0$-trivial, we have that $\eta = z_0'$ in
$\CH_0(X_{k(X)})$, as they are both 0-cycles of degree 1. Denote by $\alpha'$ and $\alpha_0'$ the pull-backs of $\alpha$ and $\omega^*(\alpha_0)$ to $\Br(X_{k(X)})$.   
Then
$$
0 = (\alpha' - \alpha_0')(\eta - z_0') = (\alpha' - \alpha_0')(\eta) - (\alpha' - \alpha_0')(z_0') = (\alpha' - \alpha_0')(\eta)
$$
by bilinearity and since $(\alpha' - \alpha_0')(z_0')$ is the
pull-back from $\Br(k)$ to $\Br(k(X))$ of $(\alpha -
\omega^*(\alpha_0))(z_0)$, which we know is zero.  Here we needed the
fact that $z_0$ is supported on closed points whose residue
fields are separable extensions of $k$ so that $z_0$ pulled back to
$X_{k(X)}$ is supported on reduced closed points, and hence
restricting a Brauer class and pushing forward to $\Spec(k(X))$ is the
same as pairing with the underlying cycle.  



But now $(\alpha' - \alpha_0')(\eta)$ is just the
class of $\alpha - \alpha_0$, restricted to $\Br(k(X))$.  Hence
$\alpha - \alpha_0 = 0$ in $\Br(k(X))$, and since the map $\Br(X) \to
\Br(k(X))$ is injective, we have that $\alpha - \alpha_0 = 0$ in
$\Br(X)$, i.e., $\alpha$ comes from $\Br(k)$.

\providecommand{\bysame}{\leavevmode\hbox to3em{\hrulefill}\thinspace}
\providecommand{\href}[2]{#2}

\end{document}